\documentclass[a4paper,11pt]{article}
\usepackage[T1]{fontenc}
\usepackage{geometry}
\usepackage[utf8]{inputenc}
\usepackage{lmodern}

\usepackage[russian, english]{babel}

\usepackage{amsmath, amsfonts, amssymb, graphicx,amsthm, mathrsfs}
\usepackage{amsfonts,amssymb,graphicx, tikz, mathabx, epsfig}
\usepackage{tikz-cd}
\usepackage{hyperref}

\usepackage{todonotes}

\graphicspath{{pictures/}}

\newtheorem{thm}{Theorem}

\newtheorem{lem}{Lemma}

\newtheorem{cor}{Corollary}

\newtheorem*{prop*}{Proposition}
\newtheorem*{thm*}{Theorem}

\newtheorem*{rem*}{Remark}
\newtheorem*{imprem*}{Important Remark}
\newtheorem*{lem*}{Lemma}
\newtheorem*{dfn*}{Definition}
\newtheorem*{cor*}{Corollary}
\newtheorem*{probs*}{Problems}
\newtheorem*{prob*}{Problem}
\newtheorem*{problem*}{Problem}
\newtheorem*{ex*}{Example}
\newtheorem*{conj*}{Conjecture}

\def\PP{\mathbb{P}}

\DeclareMathOperator{\li}{li}

\definecolor{vio}{RGB}{118, 120, 238}

\usepackage{etoolbox}
\usepackage{hyperref}

\makeatletter
\newcounter{author}
\renewcommand*\author[1]{%
  \stepcounter{author}%
  \ifnum\c@author=1
    \gdef\@author{#1}%
  \else
    \xdef\@author{\unexpanded\expandafter{\@author\and#1}}%
  \fi
  \csgdef{author@\the\c@author}{#1}}
\newcommand*\email[1]{%
  \csgdef{email@\the\c@author}{#1}}
\newcommand*\address[1]{%
  \csgdef{address@\the\c@author}{#1}}
\AtEndDocument{%
  \xdef\author@count{\the\c@author}%
  \c@author=1
  \print@authors}
\newcommand*\print@authors{%
  \ifnum\c@author>\author@count
  \else
    \print@author{\the\c@author}%
    \advance\c@author by 1
    \expandafter\print@authors
  \fi}
\newcommand*\print@author[1]{%
  \par\medskip
  \begin{tabular}{@{}l@{}}%
   
    \csuse{address@#1}\\
    \textit{E-mail address}:
    \href{mailto:\csuse{email@#1}}{\csuse{email@#1}}
  \end{tabular}}
\makeatother

\title{An explicit version of Chen's theorem assuming the Generalized Riemann Hypothesis}
\author{Matteo Bordignon\thanks{ The research was supported by OP RDE project No.
CZ.$02.2.69/0.0/0.0/18\_053/0016976$ International mobility of research,
technical and administrative staff at Charles University.}}
\address{Charles University, Faculty of Mathematics and Physics, Department of Algebra,\\ Sokolovská 83, 186 00 Praha 8, Czech Republic}
\email{matteobordignon91@gmail.com}

\author{Valeriia Starichkova }
\address{ UNSW Canberra, Department of Science, Northcott Drive, Canberra,\\ ACT Australian Capital Territory 2612} \email{v.starichkova@adfa.edu.au}

\begin{document}

\maketitle

\begin{abstract}
We prove that assuming the Generalized Riemann Hypothesis every even integer larger than $\exp(\exp(15.85))$ can be written as the sum of a prime number and a number that has at most two prime factors.
\end{abstract}
 
\section{Introduction}

In \cite{BJS22}, Johnston and the authors of this paper recently built
upon unpublished work of Yamada \cite{Yamada} to prove an effective and explicit variant of Chen’s theorem that holds for a limited set of natural numbers. Namely, \cite[Corollary 4]{BJS22} states the following.
\begin{thm}\label{theo:Chen0}
 Every even integer $N>\exp(\exp(34.5))$ can be represented as the sum of a prime and a square-free number $\eta>1$ with at most two prime factors.
\end{thm}
It is easy to draw a comparison between this result and Goldbach's weak conjecture, also known as the ternary Goldbach problem.
\begin{thm}[Vinogradov--Helfgott]
For any odd number $N \ge 7$ there exist three primes $p_1, p_2$ and $p_3$, such that
\begin{equation*}
N=p_1+p_2+p_3.
\end{equation*}
\end{thm}
This result was first proved by Vinogradov \cite{Vin} with a non explicit constant and, after different papers in this direction, proved for the conjectured range of natural numbers by Helfgott \cite{Helf}. We are mainly interested in two notable papers. Kaniecki showed in \cite{Kan} that, under the Generalized
Riemann Hypothesis (GRH), every odd integer is a sum of at most five primes. Deshouillers, Effinger,
te Riele and Zinoviev demonstrated in \cite{Des} that, under GRH,
Goldbach's weak conjecture holds for all $N > 7$. Drawing inspiration from these papers we will improve the range of Theorem \ref{theo:Chen0} assuming GRH and reworking its proof. With this aim we will use the new explicit and complete version of the prime number theorem under GRH proven by Ernvall-Hyt\"{o}nen and Paloj\"{a}rvi \cite{ErnvallPalojarvi2020}.
We will now state our main result.
\begin{thm}
\label{Theo:B.}
Let $\pi_2(N)$ denote the number of representations of a given even integer $N$ as the sum of a prime number and a product of at most two prime factors. If $N> \exp(\exp(15.85))$, then assuming GRH we have
\begin{equation*}
\pi_2(N)> 4 \cdot 10^{-4}\frac{U_N N}{\log^2N},
\end{equation*}
where, with $\gamma$ the Euler--Mascheroni constant,
\begin{equation*}
U_N=2 e^{\gamma}\prod_{p>2}\left( 1-\frac{1}{(p-1)^2}\right)\prod_{p>2,p|N}\frac{p-1}{p-2}.
\end{equation*}
\end{thm}
There are two main reasons why the assumption of GRH allows to improve upon the unconditional result from Theorem \ref{theo:Chen0}: the potential Siegel zeros contribute significantly to the error terms in \cite{BJS22}, thus, the GRH simplifies and reduces the error terms we are dealing with in this work. Secondly, the error terms $r(d)$ appearing in the linear sieve (Theorem \ref{theo:JR}) can be expressed via the error terms from the prime number theorem in arithmetic progressions. The assumption of GRH lowers such error terms as shown in \cite{ErnvallPalojarvi2020}.

We will derive an easy corollary from Theorem \ref{Theo:B.}.
\begin{cor}\label{distinctcor}
    Every even integer $N>\exp(\exp(15.85))$ can be represented as the sum of a prime and a square-free number $\eta>1$ with at most two prime factors.
\end{cor}
We should note that the main obstruction at proving Corollary \ref{distinctcor} for smaller $N$ can be found in the explicit version of the linear sieve, this suggests to the researchers interested in improving the explicit version of Chen's theorem to focus on improving this result. Another obstruction, and thus a point of interest to improve the explict version of Chen's theorem, can be found in the size of $\epsilon(u)$ such that 
\begin{equation*}
        \prod_{u\le p<z}\left(1-\frac{1}{p-1}\right)^{-1}<\left(1+\epsilon(u)\right)\frac{\log z}{\log u}.
    \end{equation*}
see Lemma \ref{lem.2} for the bound we prove for such function.
\par
An outline of the paper is as follows. In Section \ref{section:PNT} we state several lemmas from the existing literature that will be used in the later sections of the paper. In Section \ref{section:LS} we introduce an explicit version of the linear sieve and some other preliminary results and definitions. In Section \ref{section:Bil} we prove an explicit bound for a bilinear form assuming GRH. 
In Section \ref{sec: ps}, we first set up all the required preliminaries for sieve methods and introduce some related lemmas. In Subsections \ref{subsec: proof-A}, \ref{subsec: proof-A_q} and \ref{subsection:B} we obtain upper and lower bounds for the sifted integer sets. We conclude in Subsections \ref{subsection:theo} and \ref{subsection:cor} by proving Theorem \ref{Theo:B.} and Corollary \ref{distinctcor}.

\section{General lemmas}
\label{section:PNT}
We require some lemmas in the arguments that follow. Lemmas \ref{lem: theta-bound} and \ref{lem: pi-li under GRH} are conditional on GRH, Lemmas \ref{lem: omega} and \ref{lem: sum-mu} could have been slightly improved under GRH, though the improvement would not affect the final result much. Here and below $p$ will denote a prime number.
\begin{lem} \cite[Corollary 1]{Schoenfeld1976} \label{eq:pi}
If the Riemann hypothesis (RH) holds, then
\begin{equation*}
|\pi(x) - \li(x)| \leq \frac{\sqrt{x} \log x}{8 \pi}, \quad \text{for} \quad 3 \leq x,
\end{equation*}
where
$$\li(x) = \int_0^{\infty}\frac{dt}{\log t}.$$
\end{lem}

\begin{lem} \cite[Theorem 18]{RosserSchoenfeld1962} \label{lem: theta-bound} Let $\theta(x) = \sum_{p \leq x} \log p$. Then 
$$\theta(x) \leq x \quad \text{for} \quad 0 < x \leq 10^{18}.$$
\end{lem}

We now introduce the conditional result under GRH that will be fundamental in our argument. Here and below we assume $N > 4 \cdot 10^{18}$, since \cite{O_H_P_14} asserts the Chen's Theorem is true unconditionally for $N$ up to $4 \cdot 10^{18}$.

\begin{lem} \label{lem: pi-li under GRH}
Let $d$ be a positive integer and $X_2$ be a real number such that $3 \leq d \leq \sqrt{X_2}$. Assume that GRH holds, then for all $N \geq X_2$ and integer $a$ with $(a,d)=1$
\begin{align*}
    \left| \pi(N;d,a) - \frac{\li(N)}{\varphi(d)} \right| \leq q_{G}(X_2) \sqrt{N} \log N,
\end{align*}
where
\begin{align*}
    q_G(X_2) &= 0.165 + \frac{12.683}{\log X_2} + \frac{254.980}{\log^2 X_2} + \frac{2607.854}{\log^3 X_2} + \frac{11605.056}{\log^4 X_2} + \frac{1.314 \log X_2}{X_2^{1/4}} + \frac{0.092 \log \log X_2}{X_2^{1/4}}\notag\\
    &+ \frac{60.883}{X_2^{1/4}} + \frac{8.250 \log \log X_2}{X_2^{1/4} \log X_2} + \frac{939.260}{X_2^{1/4} \log X_2} - \frac{237.934}{X_2^{1/2} \log X_2} \leq 0.640 \quad \text{for} \quad X_2 \geq 4 \cdot 10^{18}.
\end{align*}
\end{lem}
\begin{proof}
Apply $q \leq \sqrt{N}$ to \cite[Theorem 1]{ErnvallPalojarvi2020}.
\end{proof}

We will use the two following unconditional results as they suffice for our aim.

\begin{lem} \cite[Theorem 11]{Guy} \label{lem: omega}
Let $\omega(N)$ count the number of prime divisors of an integer $N$ without multiplicities, then
\begin{equation}
\label{eq:w(n)}
\omega(N)< \frac{1.3841\log N}{\log \log N} \quad \text{for} \quad N \geq 3.
\end{equation}
\end{lem}

\begin{lem}[{\cite[(4.6) and Lemma 4.5]{Buthe1}}]\label{lem: sum-mu}
For $x \geq 10^9$,
\begin{align*} \label{eq:phi}
    &\sum_{n\le x} \mu^2(n) \le 0.608 x, \\
    &\sum_{n\le x}\frac{\mu^2(n)}{\varphi(n)}\le\log x + 1.333 +\frac{58}{\sqrt{x}}\le 1.1\log x.
\end{align*}
\end{lem}

\section{An explicit formula for the linear sieve}
\label{section:LS}
We now introduce the version of the linear sieve proven by Johnston and the authors in \cite[Theorem 6 \& Table 1]{BJS22}.
\begin{thm}
\label{theo:JR}
Let $A=\{a(n)\}_{n=1}^{\infty}$ be an arithmetic function such that
\begin{equation*}
a(n)\ge 0 \quad \text{for all } n, \quad \text{and} \quad 
|A|=\sum_{n=1}^{\infty} a(n)< \infty.
\end{equation*}
Let $\mathbb{P}$ be a set of prime numbers, and for $z\ge 2$ let
\begin{equation*}
P(z)=\prod_{\substack{p \in \mathbb{P}\\ p<z}}p.
\end{equation*}
Let 
\begin{equation*}
S(A,\mathbb{P}, z)=\sum_{\substack{n=1\\ (n, P(z))=1}}^{\infty}a(n).
\end{equation*}
For every $n\ge 1$, let $g_n(d)$ be a multiplicative function such that
\begin{equation*}
0\le g_n(p)< 1 \quad \text{for all } p \in \mathbb{P}.
\end{equation*}
Define $r(d)$ by
\begin{equation*}
|A_d|=\sum_{\substack{n=1\\ d|n}}^{\infty} a(n)= \sum_{n=1}^{\infty} a(n)g_n(d)+r(d).
\end{equation*}
Let $\mathbb{Q}\subseteq \mathbb{P}$, and $Q$ be the product of primes from $\mathbb{Q}$. Suppose that, for some $\epsilon$ such that $0<\epsilon< 1/63$, the inequality
\begin{equation}
\label{eq:cond1}
\prod_{\substack{p\in \mathbb{P}\setminus\mathbb{Q}\\ u \le p < z}}(1-g_n(p))^{-1} \le (1+\epsilon) \frac{\log z}{\log u},
\end{equation}
holds for all $n$ and $1<u<z$. Then, for any $D\geq z$ we have an upper bound
\begin{equation*}
S(A,\mathbb{P},z)<(F(s)+\epsilon C_{1}(\epsilon) e^2h(s))X+R,
\end{equation*}
and for any $D\ge z^2$ we have a lower bound
\begin{equation*}
S(A,\mathbb{P},z)>(f(s)-\epsilon C_{2}(\epsilon) e^2h(s))X-R,
\end{equation*}
where
\begin{equation*}
s=\frac{\log D}{\log z},
\end{equation*}
\begin{equation*}
h(s)= \begin{cases} 
      e^{-2} & 1\le s \le 2 \\
      e^{-s} & 2\le s \le 3 \\
      3s^{-1}e^{-s} & s\ge 3,
      \end{cases} 
\end{equation*}
$f(s)$ and $F(s)$ are two functions defined in \cite[(9.27) and (9.28)]{Nathanson1996}, $C_{1,2}(\epsilon)$ come from Table \ref{tab:fF},
\begin{equation*}
X=\sum_{n=1}^{\infty} a(n) \prod_{p|P(z)}(1-g_n(p)),
\end{equation*}
and the remainder term is
\begin{equation*}
R=\sum_{\substack{d|P(z)\\ d<QD}}|r(d)|.
\end{equation*}
If there is a multiplicative function $g(d)$ such that $g_n(d)=g(d)$ for all $n$, then
\begin{equation*}
X=V(z)|A|, \quad \text{where} \quad 
V(z)=\prod_{p|P(z)}(1-g(p)).
\end{equation*}
\end{thm}
\begin{table}
\centering
    \begin{tabular}{ | l | l | l |}
    \hline
    $\epsilon^{-1}$ & $C_1(\epsilon) $ & $C_2(\epsilon)$ \\  
    \hline
    $63$& $32881$&$32875$ \\
    \hline
    $64$& $7582$&$7580$ \\
    \hline
    $65$& $3890$& $3890$ \\
    \hline
    $66$& $2542$& $2542$ \\
    \hline
    $67$& $1880$& $1881$ \\
    \hline
    $68$& $1480$&$1500$ \\
    \hline
   $69$ & $1254$&$1255$ \\
    \hline
    $70$&$1084$&$1086$ \\
    \hline
    $71$&$960$&$962$\\
    \hline
    $72$&$865$&$867$ \\
     \hline
    $73$&$790$&$791$ \\
     \hline
    $74$&$729$&$730$ \\
    \hline
    \end{tabular}
\begin{tabular}{ | l | l | l |}
    \hline
    $\epsilon^{-1}$ & $C_1(\epsilon) $ & $C_2(\epsilon)$  \\  
    \hline
     $75$& $678$& $679$\\
    \hline
    $76$& $635$& $636$ \\
    \hline
    $77$& $598$&$600$ \\
    \hline
    $78$ & $566$&$568$ \\
    \hline
     $79$& $538$&$540$ \\
    \hline
    $80$&$514$&$515$ \\
    \hline
   $81$& $492$&$493$ \\
    \hline
    $84$& $438$&$439$\\
    \hline
    $87$& $398$&$400$ \\
    \hline
    $93$& $341$&$343$ \\
     \hline
    $99$& $303$& $305$ \\
     \hline
    $114$& $247$& $249$ \\
    \hline
    \end{tabular}
    \begin{tabular}{ | l | l | l |}
    \hline
    $\epsilon^{-1}$ & $C_1(\epsilon) $ & $C_2(\epsilon)$ \\  
    \hline
    $143$& $198$& $200$\\
    \hline
    $200$& $162$&$164$\\
    \hline
    $249$&$149$&$150$\\
    \hline
    $300$&$141$&$142$\\
    \hline
    $400$&$132$&$134$\\
    \hline
    $500$&$127$&$129$\\
    \hline
   $600$& $124$& $126$\\
    \hline
   $700$&$122$&$124$\\
    \hline
   $800$&$121$&$122$\\
    \hline
   $900$&$120$&$121$\\
     \hline
    $1000$&$119$&$120$\\
     \hline
    $1100$&$118$&$120$\\
    \hline
    \end{tabular}
    \begin{tabular}{ | l | l | l |}
    \hline
    $\epsilon^{-1}$ & $C_1(\epsilon) $ & $C_2(\epsilon)$  \\  
    \hline
    $1200$&$117$&$119$\\
    \hline
    $1400$&$117$&$118$\\
    \hline
    $1500$&$116$&$118$\\
    \hline
    $1600$&$116$&$117$\\
    \hline
    $1800$&$115$&$117$\\
    \hline
    $2100$&$115$&$116$\\
    \hline
   $2300$&$114$&$116$\\
    \hline
    $3300$&$113$&$115$\\
    \hline
    $4500$&$113$&$114$\\
    \hline
    $6100$&$112$&$114$\\
     \hline
    $12200$&$112$&$113$\\
     \hline
    $39500$&$111$&$113$\\
    \hline
    \end{tabular}    
\caption{Values for $C_1(\epsilon)$ and $C_2(\epsilon)$}  
\label{tab:fF}  
\end{table}
For our applications we will choose $g_n$ such that $g_n(p)=\frac{1}{p-1}$. With this choice we will provide the value of $\epsilon$, for which \eqref{eq:cond1} is satisfied under certain conditions. The smaller $\epsilon$ we will be able to take, the better lower bound for $N$ we will achieve. Lemma \ref{lem.2} provides the value of $\epsilon$ we will use in the paper. Any improvement of the lemma would lead to a strong improvement in the range of $N$ in Theorem~\ref{Theo:B.} and it is thus something that would be interesting to pursue with a computational approach.

\begin{lem}\label{betrandlem}
    For all $x\ge 9551$, there exists a prime in the interval $[0.996x,x]$.
\end{lem}
\begin{proof}
We use \cite[Table 8]{O_H_P_14} for $9551 \leq x \leq 4\cdot 10^{18}$ and \cite[Theorem 1.1]{KadiriLumley2014} for $x > 4\cdot 10^{18}$.
\end{proof}
By increasing the value of $x$ one can obtain much stronger results than Lemma \ref{betrandlem}. See \cite{KadiriLumley2014} and \cite{CH_L_21} for the recent improvements on the length of the intervals containing primes.

\begin{lem}\label{mertenlem-uncond}
    We have
    \begin{align}
    &\sum_{p<x}\frac{1}{p}\ge\log\log x+M, \quad \text{for} \quad 3 \leq x \leq 10^8,\nonumber\\
    &\sum_{p<x}\frac{1}{p}\ge\log\log x+M - \frac{1.4998 \cdot 10^{-4}}{\log x}, \quad \text{for} \quad x > 10^8,\nonumber\\
    &\sum_{p\le x}\frac{1}{p} \leq \log\log x+M+\frac{4.47 \cdot 10^{-9}}{\log x}, \quad \text{for} \quad x > \exp(1000). \label{big1p}
    \end{align}
\end{lem}

\begin{proof}
    By \cite[(4.20)]{RosserSchoenfeld1962} we have
    \begin{equation}\label{1pinteq}
        \sum_{p\le x}\frac{1}{p}=\log\log x+M+\frac{\theta(x)-x}{x\log x}+\int_x^\infty\frac{(y-\theta(y))(1+\log y)}{y^2\log^2y}\mathrm{d}y,
    \end{equation}
    where $\theta$ is Chebyshev's theta function. To obtain \eqref{big1p}, we note that
    \begin{equation*}
        \sum_{p\le x}\frac{1}{p} \leq \log\log x+M+\frac{M_1}{\log x} \left( 1 + \frac{3}{2 \log x} \right),
    \end{equation*}
    where $M_1$ is from \cite[Table 15]{BroadbentSamuelKadiri2021} and depends on the lower bound for $x$. In particular, $x > \exp(1000)$ gives
    \begin{equation*}
        \sum_{p\le x}\frac{1}{p} \leq \log\log x+M+\frac{4.47 \cdot 10^{-9}}{\log x}.
    \end{equation*}
    
    The lower bound for $3 \leq x \leq 10^8$ follows from \cite[Theorem 20]{RosserSchoenfeld1962}. To prove the lower bound for $x>10^{8}$ we use again \cite[Table 15]{BroadbentSamuelKadiri2021} for $x = 10^8$ and get
    \begin{equation} \label{eq: bound-theta}
        \left|\frac{\theta(x)-x}{x\log x}\right|\leq\frac{2.7457\cdot 10^{-3}}{\log^2x} \leq \frac{1.491 \cdot 10^{-4}}{\log x},
    \end{equation}
for $x > 10^8$. We now split into the cases $x\leq 10^{19}$ and $x>10^{19}$. In the first case, we have
\begin{align*}
    \int_x^\infty\frac{(y-\theta(y))(1+\log y)}{y^2\log^2y}\mathrm{d}y
    &=\int_x^{10^{19}}\frac{(y-\theta(y))(1+\log y)}{y^2\log^2y}\mathrm{d}y+\int_{10^{19}}^\infty\frac{(y-\theta(y))(1+\log y)}{y^2\log^2y}\mathrm{d}y.
\end{align*}
The first integral on the right-hand side is non-negative by \cite[Theorem 2]{Buthe2018}.
For the second integral we again use \cite[Table 15]{BroadbentSamuelKadiri2021} to obtain
\begin{align*}
    \left|\int_{10^{19}}^\infty\frac{(y-\theta(y))(1+\log y)}{y^2\log^2y}\mathrm{d}y\right|&\leq 8.6315\cdot 10^{-7}\left(\frac{1}{2\log^2(10^{19})}+\frac{1}{\log(10^{19})}\right) \\
    &\leq 8.6315\cdot 10^{-7}\left(\frac{1}{2\log(10^{19})}+1\right) \frac{1}{\log x} \\
    & \leq \frac{8.73016 \cdot 10^{-7}}{\log x}.
\end{align*}
for $x \leq 10^{19}$. We combine the last bound with \eqref{eq: bound-theta} to get the lower bound for $x > 10^8$.

To get the lower bound for $x > 10^{19}$ we just note that by \cite[Table 15]{BroadbentSamuelKadiri2021},
\begin{align*}
    \left|\int_{x}^\infty\frac{(y-\theta(y))(1+\log y)}{y^2\log^2y}\mathrm{d}y\right|&\leq 8.6315\cdot 10^{-7}\left(\frac{1}{2\log^2 x}+\frac{1}{\log x}\right) \\
    &\leq 8.6315\cdot 10^{-7}\left(\frac{1}{2\log(10^{19})}+1\right) \frac{1}{\log x} \\
    & \leq \frac{8.73016 \cdot 10^{-7}}{\log x}.
\end{align*}
\end{proof}

It is important to mention that the unconditional bounds above, being based on the computational results from \cite{BroadbentSamuelKadiri2021}, are better for smaller values of $x$ than the conditional bounds \cite[Corollary 2]{Schoenfeld1976} as this last result aims to give an asymptotic explicit improvement which isn't necessarily optimal in the initial range. While it should be possible to blend these two results together we believe this would not give a strong improvement in the range we are interested in, and we will thus not further pursue this idea.

\begin{lem} \label{lem.2}
Let $z>\exp(1000)$. Then for all $9551\leq u<z$ we have
    \begin{equation*}
        \prod_{u\le p<z}\left(1-\frac{1}{p-1}\right)^{-1}<\left(1+\epsilon(u)\right)\frac{\log z}{\log u},
    \end{equation*}
    with
    \begin{equation} \label{def: epsilon}
        \epsilon(u) = \left( 1 + \frac{1}{0.996 u - 1} \right) \left(1+\frac{1.5\cdot 10^{-4}}{\log u}+\frac{2.25\cdot 10^{-8}}{\log^2u}\right)\left(1+\frac{1}{u-1}+\frac{1}{(u-1)^2}\right) - 1.
    \end{equation}
\end{lem}
\begin{proof}
    We first note that
    \begin{equation*}
        \prod_{u\le p<z}\left(1-\frac{1}{p-1}\right)^{-1}=\prod_{u\le p<z}\left(\frac{(p-1)^2}{p(p-2)}\right)\prod_{u\le p<z}\left(1-\frac{1}{p}\right)^{-1}.
    \end{equation*}
    By Lemma \ref{betrandlem} we then have
    \begin{align*}
        \prod_{u\le p<z}\left(\frac{(p-1)^2}{p(p-2)}\right)&=\prod_{u\le p<z}\left(1+\frac{1}{p(p-2)}\right)\\
        &\le\prod_{0.996u\le p<z}\left(1+\frac{1}{p^2}\right)\\
        &\le\prod_{0.996u\le p<z}\left(1+\frac{1}{p^2}\right)\\
        &\le 1+\sum_{n\ge 0.996u}\frac{1}{n^2}\le 1+\frac{1}{0.996u-1}.
    \end{align*}
    Thus,
    \begin{equation}\label{prodeq1}
        \prod_{u\le p<z}\left(1-\frac{1}{p-1}\right)^{-1}<\left(1+\frac{1}{0.996u-1}\right)\prod_{u\le p<z}\left(1-\frac{1}{p}\right)^{-1}.
    \end{equation}
    Next, we note that
    \begin{equation}\label{prodtoexp}
        \prod_{u\le p<z}\left(1-\frac{1}{p}\right)^{-1}=\exp\left(-\sum_{u\le p<z}\log\left(1-\frac{1}{p}\right)\right).
    \end{equation}
    Now, by Lemma \ref{mertenlem-uncond}
    \begin{equation}\label{merteneq}
        \sum_{u\le p< z}\frac{1}{p}\le \sum_{p<z}\frac{1}{p}-\sum_{p<u}\frac{1}{p}\le\log\log z-\log\log u+\frac{1.5\cdot 10^{-4}}{\log u}.
    \end{equation}
    Using \eqref{prodtoexp}, \eqref{merteneq} and that for $0<x\le 1/2$,
    \begin{equation*}
        \log(1-x)\ge-x-x^2,\quad e^x\le 1+x+x^2,
    \end{equation*}
    we have,
    \begin{align}
        \prod_{u\le p<z}\left(1-\frac{1}{p}\right)^{-1}&\le\frac{\log z}{\log u}\exp\left(\frac{1.5\cdot 10^{-4}}{\log u}\right)\exp\left(\sum_{p\ge u}\frac{1}{p^2}\right)\notag\\
        &\le\frac{\log z}{\log u}\left(1+\frac{1.5\cdot 10^{-4}}{\log u}+\frac{2.25\cdot 10^{-8}}{\log^2u}\right)\exp\left(\sum_{n\ge u}\frac{1}{n^2}\right)\notag\\
        &\le\frac{\log z}{\log u}\left(1+\frac{1.5\cdot 10^{-4}}{\log u}+\frac{2.25\cdot 10^{-8}}{\log^2u}\right)\left(1+\frac{1}{u-1}+\frac{1}{(u-1)^2}\right).\label{prodeq2}
    \end{align}
    Using \eqref{prodeq1}, \eqref{prodeq2} and $u>10^{8}$ then gives the desired result.
\end{proof}

\section{An explicit bilinear form assuming GRH}
We modify the explicit upper bound for the bilinear form from \cite{BJS22} under GRH.
\label{section:Bil}
\begin{lem}[{\cite[Lemma 31]{BJS22}}]
\label{lem:bil}
Let $a(n)$ be an arithmetic function with $|a(n)| \leq 1$ for all $n$. Let $X, Y, Z \ge 25$ be real numbers with $Z < Y$, and $q_G$ be as in Lemma \ref{lem: pi-li under GRH}. Let $10^9 \leq D^*$ be such that
\begin{equation*}
    \left( \frac{\sqrt{Y}\log D^*}{q_G(Y) \log Y \log 2}\right)^{2/3} \leq Z,
\end{equation*}
then for all even $a \in \mathbb{Z}$
\begin{equation} \label{eq: bilinear form}
    \sum_{\substack{d < D^{*} \\ (a,d)=1}} \mu^2(d) \max \left| \sum_{n < X} \sum_{\substack{Z \leq p < Y \\ np \equiv a (\text{mod } d)}} a(n) - \frac{1}{\varphi(d)}\sum_{n < X} \sum_{\substack{Z \leq p < Y \\ (np, d) = 1}} a(n)\right| \leq m(X,Y,D^*)\frac{XY \log^{5/3} D^* \log^{1/3} Y}{Y^{1/6}},
\end{equation}
where
\begin{align*}
    m(X,Y,D^*) &= 2.81 ~q_G^{1/3}(Y)
    + \frac{2.809}{q_G^{2/3}(Y) \sqrt{Y} \log Y} + \frac{1.721 \log^{2/3} D^*}{q_G^{1/3}(Y) Y^{2/3} \log^{2/3} Y\log \log D^*} \\
    &+ 5.498 \left( \frac{1}{\sqrt{X}} + \frac{1}{\sqrt{Y}} \right) \frac{Y^{1/6} \log^{1/3} D^*}{\log^{1/3} Y} + 19.044 \frac{D^* \log^{1/3} D^*}{X Y^{5/6} \log^{1/3} Y}.
\end{align*}
We note that $m(X,Y,D^*)$ is decreasing in $X, Y$ and increasing in $D^*$.
\end{lem}
\begin{proof}
By the orthogonality property of the Dirichlet characters
\begin{align*}
    \sum_{n < X} \sum_{\substack{Z \leq p < Y \\ np \equiv a (\text{mod } d)}} a(n) &= \sum_{n < X} \sum_{\substack{Z \leq p < Y}} \frac{a(n)}{\varphi(d)} \sum_{\chi \text{ (mod $d$)}} \overline{\chi}(a)\chi(np) \\
    &= \frac{1}{\varphi(d)} \sum_{\chi \text{ (mod $d$)}} \overline{\chi}(a) \sum_{n < X} a(n)\chi(n) \sum_{Z \leq p < Y} \chi(p) \\
    &= \frac{1}{\varphi(d)} \sum_{\chi\ne \chi_{0,d} \text{ (mod $d$)}} \overline{\chi}(a) \sum_{n < X} a(n)\chi(n) \sum_{Z \leq p < Y} \chi(p) + \frac{1}{\varphi(d)}\sum_{n < X} \sum_{\substack{Z \leq p < Y \\ (np, d) = 1}} a(n),
\end{align*}
so the bilinear form is bounded by
\begin{align*}
    \frac{\mu^2(d)}{\varphi(d)} \sum_{\chi\ne \chi_{0,d} \text{ (mod $d$)}} \left| \sum_{n < X} a(n)\chi(n) \right| \left| \sum_{Z \leq p < Y} \chi(p) \right|.
\end{align*}
Every non-trivial character $\chi$ mod $d$ can be uniquely factorized into $\chi=\chi_{0,s} \chi_1$, with $d = sr$, $r \ne 1$, $\chi_{0,s}$ the principal character mod $s$ and $\chi_1$ a primitive character mod $r$. Thus the upper bound above can be rewritten as follows
\begin{align}\label{eq:rs}
\sum_{\substack{rs<D^* \\ r \ne 1 \\ (rs,a)=1}} \frac{\mu^2(rs)}{\varphi(sr)}~~\sideset{}{^{*}}\sum_{\chi \text{ (mod $r$})}\left|\sum_{\substack{n<X\\ (n,s)=1}}a(n) \chi(n) \right|\left|\sum_{\substack{Z\le p <Y \\ p\nmid s}}\chi(p) \right|,
\end{align}
where $*$ means that the sum is restricted to primitive characters. Let us note that conditions $r \ne 1$, $(r,a)=1$ imply $r \ge 3$ for even $a$. We begin by estimating the sum restricted to $3 \le r\leq D_0$, with $D_0 \leq \sqrt{Z}$ to be defined later. In this case
\begin{align*}
    \sum_{\substack{Z\le p <Y}}\chi(p) = \sum_{a \text{ (mod $r$)}}\chi(a) \sum_{\substack{Z \leq p < Y \\ p\equiv a \text{ (mod $r$)}}} 1 = \sum_{a \text{ (mod $r$)}}\chi(a) (\pi(Y; r, a) - \pi(Z; r, a)) + E(r),
\end{align*}
with $|E(r)| \leq 2 \varphi(r)$ covering the potential cases when $Y$ or $Z$ is prime. Thus by Lemma \ref{lem: pi-li under GRH}
\begin{align*}
    \left| \sum_{\substack{Z\le p <Y}}\chi(p) \right| &\leq
    \left| \sum_{a \text{ (mod $r$)}}\chi(a) \frac{\li(Y) - \li(Z)}{\varphi(r)} \right| + \varphi(r) \max \left\{ q_{G}(Z) \sqrt{Z} \log Z, q_{G}(Y) \sqrt{Y} \log N \right\} + 2 \varphi(r) \\
    &= \varphi(r) \left( 2 +\max \left\{ q_{G}(Z) \sqrt{Z} \log Z, q_{G}(Y) \sqrt{Y} \log Y \right\}\right) \\
    &\le \varphi(r) \left( 2 + q_{G}(Y) \sqrt{Y} \log Y \right) \quad \text{for} \quad 25 \le Z < Y,
\end{align*}
hence
\begin{align*}
    \sideset{}{^{*}}\sum_{\chi \text{ (mod $r$)}}\left|\sum_{\substack{Z\le p <Y \\ p\nmid s}}\chi(p) \right|  &\le \varphi^2(r) \left( 2 +  q_{G}(Y) \sqrt{Y} \log Y \right) + \sideset{}{^{*}}\sum_{\chi \text{ (mod $r$)}} \omega(s) \\
    &\le D_0 \left( \frac{1.3841 \log D^{*}}{\log \log D^{*}} + 2D_0 + D_0 q_G(Y) \sqrt{Y} \log Y\right).
\end{align*}
By Lemma \ref{lem: sum-mu} we get
\begin{align}
\sum_{\substack{rs<D^*\\ r \leq D_0 }}\frac{\mu^2(rs)}{\varphi(rs)}~~&\sideset{}{^{*}}\sum_{\chi \text{ (mod $r$)}}\left|\sum_{\substack{n<X\\ (n,s)=1}}a(n) \chi(n) \right|\left|\sum_{\substack{Z\le p <Y \\ p\nmid s}}\chi(p) \right| \\
& \le D_0 X \left( \frac{1.3841 \log D^{*}}{\log \log D^{*}} + 2D_0 + D_0 q_G(Y) \sqrt{Y} \log Y \right) \left(\sum_{l\le D^*}\frac{\mu^2(l)}{\varphi(l)}\right) \nonumber\\
& \le 1.1 D_0 X \log D^{*} \left( \frac{1.3841 \log D^{*}}{\log \log D^{*}} + 2D_0 + D_0 q_G(Y) \sqrt{Y} \log Y \right). \label{eq:rs1}
\end{align}
We are now left with estimating the sum in \eqref{eq:rs} restricted to $r \ge D_0$, which can be bounded by
\begin{equation} \label{eq: bil-form-r-large}
\sum_{\substack{s<D^* }}\frac{\mu^2(s)}{\varphi(s)} \sum_{\substack{ D_0 < r < D^*}}\frac{1}{\varphi(r)}~~\sideset{}{^{*}}\sum_{\chi \text{ (mod $r$)}}\left|\sum_{\substack{n<X\\ (n,s)=1}}a(n) \chi(n) \right|\left|\sum_{\substack{Z\le p <Y \\ p\nmid s}}\chi(p) \right|.
\end{equation} 
To do so, we divide the interval $D_0\le r \le D^*$ into dyadic subintervals
\begin{equation*}
D_k \le r \le 2D_k, \quad \text{where} \quad D_k=2^kD_0, \quad 0 \le k \le \frac{\log (D^*/D_0)}{\log 2}.
\end{equation*}
Using Cauchy's inequality and the large sieve inequality \cite[Theorem 4, p.~160]{Davenport} we obtain
\begingroup
\allowdisplaybreaks
\begin{align}
&\sum_{\substack{ D_k\le r < 2D_k }}\frac{1}{\varphi(r)}~~\sideset{}{^{*}}\sum_{\chi \text{ (mod $r$)}}\left|\sum_{\substack{n<X\\ (n,s)=1}}a(n) \chi(n) \right|\left|\sum_{\substack{Z\le p <Y \\ p\nmid s}}\chi(p) \right| \nonumber \\&
\le \frac{1}{D_k}\sum_{\substack{ D_k\le r < 2D_k }}~~ \sideset{}{^{*}}\sum_{\chi \text{ (mod $r$)}}\left( \frac{r}{\varphi(r)} \right)^{1/2}\left|\sum_{\substack{n<X\\ (n,s)=1}}a(n) \chi(n) \left( \frac{r}{\varphi(r)} \right)^{1/2} \right|\left|\sum_{\substack{Z\le p <Y \\ p\nmid s}}\chi(p) \right| \nonumber \\&
\le \frac{1}{D_k}\left( \sum_{D_k\le r < 2D_k}\frac{r}{\varphi(r)}~~\sideset{}{^{*}}\sum_{\chi \text{ (mod $r$)}} \left|\sum_{\substack{n<X\\ (n,s)=1}}a(n) \chi(n) \right|^2\right)^{\frac{1}{2}} 
\cdot \left( \sum_{D_k\le r < 2D_k} \frac{r}{\varphi(r)}~~\sideset{}{^{*}}\sum_{\chi \text{ (mod $r$)}} \left|\sum_{\substack{Z\le p <Y \\ p\nmid s}}\chi(p) \right|^2\right)^{\frac{1}{2}} \nonumber\\&
\le \frac{1}{D_k}\left((X+12D_k^2)(Y+12D_k^2)XY \right)^{\frac{1}{2}} \nonumber \\&
\le \left( 144 (D^*)^2 + 12 (X + Y) + \frac{XY}{D_0^2} \right)^{1/2} (XY)^{1/2} \nonumber\\&
\le \left( 12 D^* + 2\sqrt{3} (\sqrt{X} + \sqrt{Y})\sqrt{XY} + \frac{XY}{D_0} \right). \label{eq:dyadic}
\end{align}
\endgroup

Using the bound above for $0 \le k \le \frac{\log (D^*/D_0)}{\log 2} $ and Lemma \ref{lem: sum-mu}, we can bound \eqref{eq: bil-form-r-large} as follows 
\begin{align}
&\sum_{\substack{s<D^* }}\frac{\mu^2(s)}{\varphi(s)} \sum_{\substack{D_0<r<D^* }}\frac{1}{\varphi(r)}~~ \sideset{}{^{*}}\sum_{\chi \text{ (mod $r$)} }\left|\sum_{\substack{n<X\\ (n,s)=1}}a(n) \chi(n) \right|\left|\sum_{\substack{Z\le p <Y \\ p\nmid s}}\chi(p) \right| \nonumber \\& 
\le \frac{1.1}{\log 2} \log^2 D^* \left( 12 D^* + 2\sqrt{3} (\sqrt{X} + \sqrt{Y})\sqrt{XY} + \frac{XY}{D_0} \right). \label{eq:rs2}
\end{align}
We choose the optimal value for $D_0$ by making the asymptotically main terms of \eqref{eq:rs1} and \eqref{eq:rs2} equal, and get
\begin{equation*}
    D_0=\left( \frac{\sqrt{Y}\log D^*}{q_G(Y) \log Y \log 2}\right)^{1/3}.
\end{equation*}

Then \eqref{eq:rs1} is less than
\begin{align*}
    \frac{XY \log^{5/3} D^* \log^{1/3} Y}{Y^{1/6}} \Bigg(& 1.405 ~q_G^{1/3}(Y) \\
    &+ \frac{2.809}{q_G^{2/3}(Y) \sqrt{Y} \log Y} + \frac{1.721 \log^{2/3} D^*}{q_G^{1/3}(Y) Y^{2/3} \log^{2/3} Y\log \log D^*} \Bigg),
\end{align*}
whereas \eqref{eq:rs2} is bounded by
\begin{align*}
    \frac{XY \log^{5/3} D^* \log^{1/3} Y}{Y^{1/6}} \Bigg(& 1.405 ~q_G^{1/3}(Y) \\
    &+ 5.498 \left( \frac{1}{\sqrt{X}} + \frac{1}{\sqrt{Y}} \right) \frac{Y^{1/6} \log^{1/3} D^*}{\log^{1/3} Y} + 19.044 \frac{D^* \log^{1/3} D^*}{X Y^{5/6} \log^{1/3} Y} \Bigg).
\end{align*}
The bounds above with \eqref{eq:rs} allow us to complete the proof of the lemma.
\end{proof}

\section{Outline of the proof} \label{sec: ps}
We will use the ideas from \cite{BJS22} and from \cite{Nathanson1996} and keep notations form these works.

Some of the definitions that will be presented in this section were previously introduced. We decided to include them here to ease readability and make this section self-contained.\par
Fix 
\begin{equation*}
z=N^{\frac{1}{8}},\quad  y=N^{\frac{1}{3}}.
\end{equation*}
We shall consider the sets
\begin{equation*}
\mathbb{P} = \{p \nmid N\},\quad
A=\left\{ N-p:p\le N, p\in\mathbb{P} \right\},\quad A_p=\left\{ a \in A:p|a \right\} \quad 
A_d=\bigcap_{p|d} A_p,
\end{equation*}
\begin{equation*}
B=\left\{ N-p_1p_2p_3:z \le p_1<y \le p_2 \le p_3, p_1p_2p_3 <N, (p_1p_2p_3,N)=1   \right\},
\end{equation*}
where $p, p_1, p_2, p_3$ denote prime numbers.
We clearly have $\left|A\right|=\pi(N)-\omega (N)$ and $\left|A_d\right|=\pi(N;d,N)-\omega (N;d,N)$, where $\omega (n;q,a)$ denotes the number of prime factors of $n$ equal $a \pmod q$.
We also set
\begin{equation*}
P(z) = \prod_{\substack{p < z \\ p \in \mathbb{P}}} p, \quad V(z)=\prod_{p | P(z)} \left(1-\frac{1}{p-1}\right), \quad 
S(A,n)=\left| A-\bigcup_{p|n} A_p\right|.
\end{equation*}
The following result of Chen is a key component in the proof of our main theorem.
\begin{lem} \cite[Theorem 10.2]{Nathanson1996}
\label{lem:>}
\begin{equation*}
\pi_2(N)>S(A,\PP,z)-\frac{1}{2}\sum_{z \le q <y}S(A_q,\PP,z)-\frac{1}{2}S(B,\PP,y)-2N^{\frac{7}{8}}-N^{\frac{1}{3}}.
\end{equation*}
\end{lem}
Thus, to prove Theorem \ref{Theo:B.}, it suffices to give a good lower bound for $S(A,\PP,z)$ and upper bounds for $S(B,\PP,z)$ and $S(A_q,\PP,z)$ for each prime $q$ with $z \le q <y$ using the linear sieve.
We will now introduce some useful lemmas in the context of Chen's theorem.
\begin{lem}\label{vjlem}
    For $x\geq 8$ we have
    \begin{equation*}
        V(x)=\frac{U_N}{\log x}\left(1+\frac{\theta (3\log x+5)}{8 \pi \sqrt{x}}\right)\left(1+\frac{2\theta}{x}\right)\left(1+\frac{8\theta\log N}{x}\right)\left(1+\frac{\theta}{x-1}\right),
    \end{equation*}
    where $|\theta|\leq 1$ and
    \begin{equation*}
        U_N=2e^{\gamma}\prod_{p>2} \left(1-\frac{1}{(p-1)^2} \right) \prod_{\substack{p>2 \\ p|N}}\frac{p-1}{p-2}.
    \end{equation*}
    In particular, when $x=N^{1/8}\geq (4\cdot 10^{18})^{1/8}$, we have
    \begin{equation*}
        \frac{U_N}{\log x}\left(1-\frac{0.62\log N}{ N^{\frac{1}{16}}}\right)<V(x)<\frac{U_N}{\log x}\left(1+\frac{0.62\log N}{ N^{\frac{1}{16}}}\right)
    \end{equation*}
    and, when $x=N^{1/3}\geq(4\cdot 10^{18})^{1/3}$, we have
    \begin{equation*}
        \frac{U_N}{\log x}\left(1-\frac{0.06\log N}{ N^{\frac{1}{6}}}\right)<V(x)<\frac{U_N}{\log x}\left(1+\frac{0.06\log N}{ N^{\frac{1}{6}}}\right)
    \end{equation*}
\end{lem}
\begin{proof}
Follows from the proof of \cite[Theorem 10.3]{Nathanson1996}, making use of \cite[Corollary 3]{Schoenfeld1976}.
\end{proof}

\begin{lem} \label{lem: bound-A}
For all $N \geq \exp(\exp(15))$\footnote{We can assume $N > \exp(\exp(15))$ since the final bound for $N$ is higher.} we have
\begin{align*}
    & \frac{N}{\log N} \leq |A| \leq \left(1+4\cdot 10^{-7}\right)
    \cdot \frac{N}{\log N}.
\end{align*}
\end{lem}
\begin{proof}
By definition of the set $A$, for $N \geq 3$,
$$||A|-\pi(N)| \leq \omega(N) \leq \frac{1.3841\log N}{\log \log N},$$
hence by \eqref{eq:pi},
\begin{align*}
    \li(N) - \frac{\sqrt{N} \log N}{8\pi} - \frac{1.3841\log N}{\log \log N} \leq |A| \leq \li(N) + \frac{\sqrt{N} \log N}{8\pi} + \frac{1.3841\log N}{\log \log N}.
\end{align*}
The statement of the lemma now follows from the bound for $\li(N)$ below
$$1 + \frac{1}{\log N} \leq \li(N) \leq 1 + \frac{1}{\log N} + \frac{3}{\log^2 N} \quad\text{for} \quad N \geq \exp(11),$$
combined with $N \geq \exp(\exp(15))$.
\end{proof}

\subsection{Lower bound for $S(A, \PP, z)$}
\label{subsec: proof-A}

\begin{thm} \label{thm: lower-S-A}
Assume that GRH holds. Let $\alpha_1$, $N \geq X_2 \geq 4 \cdot 10^{18}$, $z = N^{1/8}$, and $y = N^{1/3}$ be such that
\begin{equation}\label{conditions-A}
    \frac{\sqrt{X_2}}{\log^{A+1} X_2}\geq 45,\quad
    \frac{N^{\alpha_1}}{\log^{A+1}N}\ge \exp\left(u\right), \quad
    0 < \alpha_1 < \frac{1}{8},
\end{equation} 
with $u \leq 10^{18}$ such that 
\begin{equation} \label{conditions-epsilon}
\epsilon := \epsilon(u) < \frac{1}{63},
\end{equation}
where $\epsilon(u)$ is defined by \eqref{def: epsilon}.
Then
\begin{align*}
    S(A,\PP,z)&>8\frac{U_N N}{\log^2 N}\left(1-\frac{0.62\log X_2}{ X_2^{\frac{1}{16}}}\right)\Bigg(\frac{2e^{\gamma}\log(3-8\alpha_1)}{4-8\alpha_1}-\epsilon C_2(\epsilon) \frac{3 e^{8\alpha_1 -2}}{4-8\alpha_1} \\
    &\qquad\qquad-\frac{1}{8}\left(1-\frac{0.62\log X_2}{ X_2^{\frac{1}{16}}}\right)^{-1}\left(2e^{\gamma}\prod_{p>2}\left(1-\frac{1}{(p-1)^2}\right)\right)^{-1}\frac{c_{G}(X_2)}{\log^{A-2}N}\Bigg),
\end{align*}
with $c_G$ defined in Lemma \ref{lem: r(d)} below, $q_G$ is defined by Lemma \ref{lem: pi-li under GRH} and $C_2(\cdot)$ from Table \ref{tab:fF}.
\end{thm}
\begin{proof}
We will apply Theorem \ref{theo:JR} with $A$, $z$, $\mathbb{P}$ defined above, and thus, with $r(d) = |A_d| - |A|/\varphi(d)$. We set
\begin{equation*}
    \mathbb{Q} := \mathbb{Q}(u) = \{ p \in \mathbb{P}, ~p < u \}, \quad
    D = N^{\frac{1}{2} - \alpha}, \quad
    g(p) = \frac{1}{p-1} \quad \text{for } p \in \mathbb{P}.
\end{equation*}

We can check that all the assumptions of Theorem \ref{theo:JR} hold: $D \geq z^2$ by \eqref{conditions-A}, and \eqref{eq:cond1} holds with $\epsilon = \epsilon(u)$ by \eqref{conditions-epsilon} and \ref{lem.2}.
Hence, we get
\begin{align*}
    S(A, \PP, z) &> (f(s) - \epsilon C_2(\epsilon) e^2 h(s))V(z)|A| - \sum_{\substack{d | P(z) \\ d < QD}} |r(d)| \\ 
    & \geq \frac{8 U_N N}{\log^2 N} \left( 1 - \frac{0.62\log N}{ N^{\frac{1}{16}}} \right) (f(s) - \epsilon C_2(\epsilon) e^2 h(s)) - \sum_{\substack{d | P(z) \\ d < QD}} |r(d)| \\
    & \geq \frac{8 U_N N}{\log^2 N} \left( 1 - \frac{0.62\log N}{ N^{\frac{1}{16}}} \right) \left(\frac{2e^{\gamma}\log(3-8\alpha_1)}{4-8\alpha_1} - \epsilon C_2(\epsilon) \frac{3 e^{8\alpha_1 -2}}{4-8\alpha_1}\right) - \sum_{\substack{d | P(z) \\ d < QD}} |r(d)|.
\end{align*}
In the last line we used $s = \frac{\log D}{\log z} = 4 - 8 \alpha_1 \in [3, 4]$, so that $f(s) = \frac{2 e^{\gamma} \log(s-1)}{s}$ and $h(s) = 3 s^{-1} e^{-s}$. It remains to bound the error term. We note that by Lemma \ref{lem: theta-bound} and \eqref{conditions-A},
\begin{equation*}
    QD \leq N^{\frac{1}{2} - \alpha_1} \prod_{p < u} p \leq N^{\frac{1}{2} - \alpha_1} \exp(\theta(u)) \leq N^{\frac{1}{2} - \alpha_1} \exp\left(u\right) \leq \frac{\sqrt{N}}{\log^{A+1} N},
\end{equation*}
so Lemmas \ref{lem: bound E_pi} and \ref{lem: r(d)} below complete the proof of the theorem.
\end{proof}
For the next lemma we define
$$E_{\pi}(x;k,l)=\pi(x;k,l)-\frac{\pi(x)}{\varphi(k)}.$$

\begin{lem}\label{lem: bound E_pi}
    Suppose $N \geq X_2 \geq 4\cdot 10^{18}$, $A \geq -1$, and $H = \frac{\sqrt{N}}{\log^{A+1} N}$. Then
    \begin{equation*}
        \sum_{\substack{d\le H\\(d,N)=1}}\mu^2(d)|E_{\pi}(N;d,N)|\leq \frac{p_G(X_2)N}{\log^A N}
    \end{equation*}
    with
    \begin{align*}\label{peq}
        p_G(X_2)=0.65(0.02 + q_G(X_2)) \leq 0.429 \quad \text{for} \quad X_2 \geq 4\cdot 10^{18},
    \end{align*}
    and $q_G(X_2)$ defined in Lemma \ref{lem: pi-li under GRH}.
\end{lem}
\begin{proof}
If $d = 1$, then $|E_{\pi}(N;d,N)| = 0$, so the upper bound from the statement holds. Suppose $d > 1$, then $3 \leq d$ since $(d, N) = 1$ and $N$ is even. By the triangle inequality
\begin{align*}
    |E_{\pi}(N;d,N)| &\leq \left| \pi(N;d,N) - \frac{\li(N)}{\varphi(d)} \right| + \frac{1}{\varphi(d)} |\li(N) - \pi(N)|,
\end{align*}
where the first term can be bounded by Lemma \ref{lem: pi-li under GRH} and the bound for the second one is provided by \cite[Corollary 1]{Schoenfeld1976}
\begin{equation*}
    \frac{1}{\varphi(d)} |\li(N) - \pi(N)| \leq \frac{1}{16\pi} \sqrt{N} \log N \quad \text{for} \quad 2.657 \leq N.
\end{equation*}
Thus,
\begin{align*}
    |E_{\pi}(N;d,N)| \leq (0.02 + q_G(X_2)) \sqrt{N} \log N,
\end{align*}
which leads to the following bound
\begin{align*}
    \sum_{\substack{d\le H\\(d,N)=1}}\mu^2(d)|E_{\pi}(N;d,N)| &\leq 0.65 H (0.02 + q_G(X_2)) \sqrt{N} \log N\\ 
    & \leq 0.65 (0.02 + q_G(X_2)) \frac{N}{\log^A N},
\end{align*}
by Lemma \ref{lem: sum-mu}.
\end{proof}

\begin{lem} \label{lem: r(d)}
Suppose all the conditions from Lemma \ref{lem: bound E_pi} are satisfied. Then
\begin{equation}\label{unsharpeq}
    \sum_{\substack{d\leq H\\ (d,N)=1}}\mu^2(d)|r(d)|\leq\frac{c_{G}(X_2)N}{\log^A N},
\end{equation}
with
\begin{align*}
    c_{G}(X_2)&=p_G(X_2)+\frac{0.9}{\sqrt{X_2} \log \log X_2} \leq 0.429 \quad \text{for} \quad X_2 \geq 4\cdot 10^{18}.
\end{align*}
\end{lem}
\begin{proof}
We note that
$$r(d) = |A_d| - \frac{|A|}{\varphi(d)},$$
hence
$$|r(d)| \leq  |E_{\pi}(N;d,N)| + \omega(N).$$
From Lemmas \ref{lem: omega}, \ref{lem: sum-mu} and \ref{lem: bound E_pi}, we have
\begin{align*}
    \sum_{\substack{d\leq H\\ (d,N)=1}}\mu^2(d)|r(d)| &\leq \frac{p_G(X_2)N}{\log^A N} + 0.65 \frac{\sqrt{N}}{\log^{A+1} N} \cdot \frac{1.3841 \log N}{\log \log N}\\
    & \leq \frac{N}{\log^A N} \left( p_G(X_2) + \frac{0.9}{\sqrt{X_2} \log \log X_2} \right).
\end{align*}
\end{proof}

\subsection{An upper bound for $\sum_{z \le q < y} S(A_q,\PP,z)$} \label{subsec: proof-A_q}
\begin{thm} \label{thm: upper-S-A-q}
Assume that GRH holds. Let $\alpha_2$, $N \geq X_2 \geq \exp(\exp(15))$, $z = N^{1/8}$, and $y = N^{1/3}$ be such that
\begin{equation}\label{conditions-A-q}
    \frac{\sqrt{X_2}}{\log^{A+1} X_2}\geq 10^9,\quad
    \frac{N^{\alpha_2}}{\log^{A+1}N}\ge \exp\left(u\right), \quad
    0 < \alpha_2 < \frac{1}{24},
\end{equation}
with $u \leq 10^{18}$ such that 
\begin{equation} \label{conditions-epsilon-q}
\epsilon := \epsilon(u) < \frac{1}{63},
\end{equation}
where $\epsilon(u)$ is defined by \eqref{def: epsilon}.
Let $$k_x = 8\left( \frac{1}{6} - x\right),$$
then
\begin{align*}
\sum_{\substack{z \le q < y\\ q\nmid N}} S(A_q,\PP,z) \leq \frac{8 U_N N}{\log^2 N} \left( l_1(X_2) + \frac{l_2(X_2)}{\log^{A-3} N} + \frac{l_3(X_2)}{\log^{A-1}N} \right).
\end{align*}
where
\begin{align*}
l_1(X_2) &= (1+3\cdot10^{-7})\left( 1+\frac{0.62\log X_2}{ X_2^{\frac{1}{16}}}\right) \frac{X_2^{1/8}}{X_2^{1/8}-1}\Bigg(\frac{e^{\gamma}}{4}\Bigg(\frac{\log(6)+\log\left(\frac{3-8\alpha_2}{3-18\alpha_2}\right)}{\left(\frac{1}{2}-\alpha_2\right)}\\
&+\frac{512}{k_{\alpha_2}\log^2 X_2}\Bigg)+\left(\log \frac{8}{3}+\frac{64}{\log^2 X_2}\right)\epsilon C_{1}(\epsilon) e^2 h(k_{\alpha_2}) \Bigg), \\
l_2(X_2) &= \frac{c_G(X_2)}{8} \left( 0.55 + \frac{1}{\log X_2} \right)\left(2e^{\gamma}\prod_{p>2}\left(1-\frac{1}{(p-1)^2}\right)\right)^{-1}, \\
l_3(X_2) &= \left( 1+\frac{0.62\log X_2}{ X_2^{\frac{1}{16}}}\right) \left( \frac{2 e^{\gamma}}{k_{\alpha}}
+ \epsilon C_{1}(\epsilon) e^2 h(k_{\alpha}) \right) c_G(X_2),
\end{align*}
with $p_G(X_2)$ defined in Lemma \ref{lem: bound E_pi} and $C_1(\cdot)$ from Table \ref{tab:fF}.
\end{thm}
\begin{proof}
Let $q$ be a prime with $z \leq q < y$. If $q | N$, then for every $N-p \in A_q$, $q$ divides both $N$ and $p$. The latter cannot be achieved since $(p,N)=1$. Thus, $S(A_q,\PP,z)=0$ for $(q,N)\ne 1$. Now we assume $(q, N)=1$.

We will proceed as in the proof of \cite[Theorem 47]{BJS22}. Namely, we apply the upper bound linear sieve to get
\begin{equation} \label{eq: up-lin-sieve-Aq}
S(A_q,\mathbb{P},z)<(F(s_q)+\epsilon C_{1}(\epsilon) e^2h(s_q))V(z)|A_q|+R_q,
\end{equation}
with
\begin{equation*}
    R_q = \sum_{\substack{d | P(z) \\ d < QD_q}} |r_q(d)|, \quad r_q(d) = |A_{qd}| - \frac{|A_q|}{\varphi(d)},
\end{equation*}
\begin{equation*}
    Q = Q(u), \quad D_q = \frac{D}{q} = \frac{N^{\frac{1}{2}-\alpha_2}}{q} \geq z, \quad \text{and} \quad s_q = \frac{\log D_q}{\log z}.
\end{equation*}

Hence, we are left with bounding
\begin{align}
\sum_{\substack{z \le q < y\\ q\nmid N}} S(A_q,\PP,z) &\leq
8U_N\left( 1+\frac{0.62\log N}{ N^{\frac{1}{16}}}\right)\sum_{\substack{z \le q < y\\q\nmid N}}\left|A_q\right|\left(\frac{F(s_q)+\epsilon C_{1}(\epsilon) e^2 h(s_q)}{\log N}\right) \label{eq: S-A-q-main} \\
&+\sum_{\substack{z \le q < y\\ q\nmid N}} R_q. \label{eq: S-A-q-error}
\end{align}
We will bound the error term \eqref{eq: S-A-q-error}. Since $d | P(z)$, $d$ is coprime to $q \geq z$, and thus $\varphi(qd) = \varphi(q) \varphi(d)$, so we have
\begin{align*}
    r_q(d) &= |A_{qd}| - \frac{|A_q|}{\varphi(d)}\\
    &=  |A_{qd}| - \frac{|A|}{\varphi(qd)} + \frac{|A|}{\varphi(qd)} - \frac{|A_q|}{\varphi(d)}\\
    &= r(qd) + \frac{r(q)}{\varphi(d)},
\end{align*}
and hence
\begin{align*}
\sum_{\substack{z \leq q < y \\ (q,N)=1}} R_q 
&\leq
\sum_{\substack{z \leq q < y \\ (q,N)=1}} \sum_{\substack{d | P(z) \\ d < QD_q}} \left| r(qd) \right| + \sum_{\substack{z \leq q < y \\ (q,N)=1}} \left| r(q) \right| \sum_{\substack{d | P(z) \\ d < QD_q}} \frac{1}{\varphi(d)} \\
&\leq \sum_{\substack{d' < QD \\ (d',N)=1}} \left| r(d') \right| + \sum_{\substack{z \leq q < y \\ (q,N)=1}} \left| r(q) \right| \sum_{d < QD} \frac{\mu^2(d)}{\varphi(d)}.
\end{align*}
To obtain the last line we note that from $qd = q' d'$ with $z \leq q, q' < y$ and $d, d' | P(z)$ we get $q = q'$ and $d = d'$.
By \eqref{conditions-A-q} and Lemma \ref{lem: theta-bound}, $QD \leq \frac{\sqrt{N}}{\log^{A+1} N}$, so
\begin{align*}
    \sum_{\substack{d' < QD \\ (d',N)=1}} \left| r(d') \right| \leq \frac{c_G(X_2) N}{\log^A N},
\end{align*}
with $c_G$ defined in Lemma \ref{lem: r(d)}. By \eqref{conditions-A-q}, $y \leq \frac{\sqrt{N}}{\log^{A+1} N}$, so we can use Lemmas \ref{lem: sum-mu} and \ref{lem: r(d)} to get
\begin{align*}
\sum_{\substack{z \leq q < y \\ (q,N)=1}} \left| r(q) \right| \sum_{d < QD} \frac{\mu^2(d)}{\varphi(d)} \leq \frac{0.55 ~c_G(X_2) N}{\log^{A-1} N}.
\end{align*}
Let us bound the main term \eqref{eq: S-A-q-main}. Since $0 < \alpha_2 < \frac{1}{24}$,
\begin{align*}
    s_q = \frac{\log D_q}{\log z} = 8 \left( \frac{1}{2} - \alpha_2 - \frac{\log q}{\log N} \right) \in [1,3],
\end{align*}
so $F(s_q) = \frac{2 e^{\gamma}}{s_q}$ by \cite[Theorem 9.8]{Nathanson1996}, hence by using $A_q = \frac{|A|}{q-1} + r(q)$ we get
\begin{align}
\sum_{\substack{z \le q < y\\q\nmid N}}\left|A_q\right|\left(\frac{F(s_q)+\epsilon C_{1}(\epsilon) e^2 h(s_q)}{\log N}\right) &\leq \sum_{\substack{z \le q < y\\q\nmid N}} \frac{|A|}{q - 1} \left( \frac{e^{\gamma}}{4 \log D/q} + \frac{\epsilon C_{1}(\epsilon) e^2 h(s_q)}{\log N} \right) \label{eq: S-A-q-error-main} \\
&+ \left( \frac{e^{\gamma}}{4 \log D/y}
+ \frac{\epsilon C_{1}(\epsilon) e^2 h(k_{\alpha})}{\log N} \right) \frac{c_G(X_2) N}{\log^A N}. \nonumber
\end{align}
We repeat the argument from the proof of \cite[Theorem 47]{BJS22} to bound the first line in \eqref{eq: S-A-q-error-main} above. By \cite[Lemma 1 (ii)]{Greaves} we get
\begin{align*}
\sum_{\substack{z \le q < y\\q\nmid N}}\frac{1}{q\log D/q}&\le \sum_{z \le q < y}\frac{1}{q\log D/q}\le 
 \int_z^y\frac{1}{t\log t \log D/t}\mathrm{d}t+\frac{64}{\log^2N}\frac{1}{\log D/y}\\ & = \frac{\log(6)+\log\left(\frac{3-8\alpha_2}{3-18\alpha_2}\right)}{\left(\frac{1}{2}-\alpha_2\right) \log N}+\frac{64}{\left( \frac{1}{2}-\frac{1}{3}-\alpha_2\right)\log^3N},
\end{align*}
where we have used a substitution $t\to N^{\tau}$ to evaluate the integral.
Using \cite[Theorem 5]{RosserSchoenfeld1962}, we have that
\begin{align*}
    \sum_{z \leq p \leq y} \frac{1}{p} < \log \log y - \log \log z + \frac{1}{\log^2 z},
\end{align*}
since $y \geq 286$ by the initial conditions on $N$. We note that $z = N^{1/8}$ would be a prime number for an even integer $N$ only if $N = 2^8 = 256$, which would contradict \eqref{conditions-A-q}. Thus, $z$ is not a prime, and
\begin{align*}
    \sum_{z < p \leq y} \frac{1}{p} < \log \log y - \log \log z + \frac{1}{\log^2 z}.
\end{align*}
Hence, we complete the proof by noting that \eqref{eq: S-A-q-error-main} is at most
\begin{align*}
\begin{split}
&\left|A\right|\frac{N^{1/8}}{N^{1/8}-1}\left(\frac{e^{\gamma}}{4}\left(\frac{\log(6)+\log\left(\frac{3-8\alpha_2}{3-18\alpha_2}\right)}{\left(\frac{1}{2}-\alpha_2\right) \log N}+\frac{512}{k_{\alpha_2}\log^3N}\right)+\left(\log \frac{8}{3}+\frac{64}{\log^2N}\right)\frac{\epsilon C_{1}(\epsilon) e^2 h(k_{\alpha_2})}{\log N} \right),
\end{split}
\end{align*}
where $|A|$ is bounded in Lemma \ref{lem: bound-A}.
\end{proof}

\subsection{An upper bound for $S(B,\PP,y)$}
\label{subsection:B}
We will now prove an upper bound for  $S(B,\PP,y)$. This is obtained using Theorem~\ref{theo:JR} together with a bound on $B$ and introducing a bilinear form that allows to bound the error term using Lemma~\ref{lem:bil}.
\begin{thm}
\label{theo:B}
Set $0<\epsilon_1<1$. Let $N$ be a positive integer, and $\alpha_3$ be a real number satisfying
\begin{equation} \label{conditions-B}
\frac{\sqrt{N}}{\log^{A+1} N}\geq 10^9, \quad \frac{N^{\alpha_3}}{ \log^{A+1} N }\ge \exp\left(u\right),\quad \quad 0 \le \alpha_3 < 1/16,
\end{equation}
with $u \leq 10^{18}$ such that 
\begin{equation} \label{conditions-epsilon-B}
\epsilon(u) < \frac{1}{63},
\end{equation}
where $\epsilon(u)$ is defined by \eqref{def: epsilon}.
Then we have 
\begin{align*}
S(B,\PP,y)&\le  \frac{U_N N}{\log^2 N}
 \Big(\left( 1+\frac{0.06\log N}{ N^{\frac{1}{6}}}\right)\\ & \cdot\left(\frac{2}{\frac{1}{2}-\alpha_3} e^{\gamma}+3\epsilon C_{1}(\epsilon)\right) 
\left(1+\epsilon_1\right)\left(\frac{\log y \left(\li (y)-\sqrt{y}\log y /(8\pi)\right)}{y}\right)\\ & \cdot\left(\overline{c} +\frac{64 \log \frac{26}{21}}{\log^2 N} +\left(\log \left(\frac{8}{3} \right)+\frac{64}{\log^2 N} \right)\left( \frac{3\log (1+\epsilon_1))}{\log N}+ \frac{27}{\log^2 N}\right)\right) 
 \\ &+\left(2 e^{\gamma}\prod_{p>2}\left( 1-\frac{1}{(p-1)^2}\right) \right)^{-1} \frac{N^{-1/48} (1 + \epsilon_1) \log^5 N}{\log(1+\epsilon_1)} \\
 & \cdot \Bigg[0.046 \cdot m\left(N^{2/3},(1+\epsilon_1)N^{1/8},\frac{\sqrt{N}}{\log^{A+1}N}\right) (1+\epsilon_1)^{-1/6} + 0.159 N^{-\frac{5}{48}} \Bigg] \Big),
\end{align*}
with $\overline{c}$ defined in Lemma~\ref{lemma:Bover}, $m(\cdot)$ defined in Lemma \ref{lem:bil}, and $C_1(\cdot)$ from Table \ref{tab:fF}.
\end{thm}
We start by recalling that
\begin{equation*}
B=\left\{ N-p_1p_2p_3:z \le p_1<y \le p_2 \le p_3,~ p_1p_2p_3 <N, (p_1p_2p_3,N)=1   \right\},
\end{equation*}
where $p_1, p_2, p_3$ denote prime numbers.
We now need to drop the restriction $(p_1,N)=1$ and relax the condition $p_1p_2p_3<N$, so that $p_1$ and $p_2p_3$ will range over independent intervals introducing a bilinear form that we will bound using Lemma~\ref{lem:bil}.

Let us fix $0 < \epsilon_1 < 1$ and let $j$ be a non-negative integer, such that
\begin{equation*}
    w_j := z(1 + \epsilon_1)^j < y,
\end{equation*}
which implies
\begin{equation*}
    j < \frac{\log(y/z)}{\log(1 + \epsilon_1)}.
\end{equation*}
We introduce sets
\begin{align*}
B^{(j)}=\{ N-p_1p_2p_3: z\le p_1<y\le p_2\le p_3, w_j p_2p_3<N,(p_2p_3,N)=1,~ w_j\le p_1<w_j(1+\epsilon_1)\},
\end{align*}
for
\begin{equation} \label{eq: bound_j_0}
0 \le j \le j_0=\left\lceil \frac{\log y/z}{\log(1+\epsilon_1)} \right\rceil - 1.
\end{equation}
Let $\overline{B}=\bigcup_j B^{(j)}$, then $B \subseteq \overline{B}$, so we have
\begin{equation}
\label{eq:numb}
S(B,\PP,y)\le S(\overline{B},\PP,y)= \sum_{j\le j_0}S(B^{(j)},\PP,y),
\end{equation}
since $B^{(j)}$ are disjoint. We will apply Theorem \ref{theo:JR} to sets $B^{(j)}$. Let us take $0 < \alpha_3 < 1/16$, $N$ and $u$ satisfying \eqref{conditions-B} and \eqref{conditions-epsilon-B}, and denote $D = N^{\frac{1}{2} - \alpha_3}$, $Q = Q(u)$, then we have
\begin{equation*}
    D \ge y, \quad
    QD \leq \frac{\sqrt{N}}{\log^{A+1} N},
\end{equation*}
since $\alpha_3 < 1/6$ and by Lemma \ref{lem: theta-bound} respectively. We set
\begin{equation}
\label{eq:D*1}
s =\frac{\log D}{\log y} = \frac{3}{2} - 3 \alpha_3.
\end{equation}

Then Theorem~\ref{theo:JR} gives
\begin{equation*}
S(B^{(j)}, \PP,y) <\left|B^{(j)}\right|V(y)(F(s)+\epsilon C_{1}(\epsilon) e^2h(s))+R^{(j)},
\end{equation*}
with
\begin{equation*}
R^{(j)}=\sum_{\substack{d<D Q\\d|P(y)}} \left|r_d^{(j)}\right|, \quad \quad r_d^{(j)}=\left|B_d^{(j)}\right|-\frac{B^{(j)}}{\varphi(d)}, \quad \text{and} \quad B_d^{(j)}=\sum_{\substack{p_1p_2p_3\equiv N \pmod d \\ z\le p_1<y\le p_2\le p_3,~ w_j\le p_1<w_j(1+\epsilon_1) \\w_j p_2p_3<N,~ (p_2p_3,N)=1 }} 1.
\end{equation*}

By Lemma~\ref{vjlem} we have $V(y)< 3\frac{U_N}{\log N}\left( 1+\frac{0.06\log N}{ N^{\frac{1}{6}}}\right)$. By the definition of $s$, $1 \le s \le 2$, and therefore $F(s)=\frac{2e^{\gamma}}{s}$ by \cite[p.~259]{Nathanson1996}. Thus, for every $j$ 
\begin{align*}
S(B^{(j)},\PP,y) < \left|B^{(j)}\right|\frac{U_N}{\log N}\left( 1+\frac{0.06\log N}{ N^{\frac{1}{6}}}\right)\left(\frac{2}{\frac{1}{2}-\alpha_3} e^{\gamma}+3\epsilon C_{1}(\epsilon) e^2h\left(\frac{3}{2}-3\alpha_3\right)\right)+R^{(j)}.
\end{align*}
Since $1 \leq \frac{3}{2} - 3\alpha_3 \leq 2$, we note that $h\left(\frac{3}{2} - 3\alpha_3\right) = e^{-2}$. Therefore, by \eqref{eq:numb},
\begin{align}
\label{eq:SBP}
\begin{split}
S(\overline{B}, \PP,y)<\left|\overline{B}\right|\frac{U_N}{\log N}\left( 1+\frac{0.06\log N}{ N^{\frac{1}{6}}}\right)\left(\frac{2}{\frac{1}{2}-\alpha_3} e^{\gamma}+3\epsilon C_{1}(\epsilon)\right)+R,
\end{split}
\end{align}
where $R=\sum_{j\le j_0} R^{(j)}$. We bound the main term first. \par
We can see that 
\begin{equation*}
\left|B^{(j)}\right|\leq(\pi(w_j(1 + \epsilon_1))-\pi(w_j) + 1)\cdot \left| \{(p_2,p_3):y\le p_2\le p_3, ~w_j p_2p_3<N,(p_2p_3,N)=1\} \right|,
\end{equation*}
\begin{equation}
\label{eq:Bin}
B \subseteq \overline{B}\subseteq \left\{ N-p_1p_2p_3:z \le p_1<y \le p_2 \le p_3,~ p_1p_2p_3 <(1+\epsilon_1)N  \right\}.
\end{equation}
We will now need an explicit upper bound for the cardinality of $\overline{B}$ that is obtained in a similar way as done by Nathanson in \cite[pp. 289--291]{Nathanson1996}. 
\begin{lem}
\label{lemma:Bover}
We have, assuming GRH and that $y\ge 2657$,
\begin{align*} 
\left| \overline{B}\right| \le & \left(1+\epsilon_1\right)\left(\frac{\log y \left(\li (y)-\sqrt{y}\log y /(8\pi)\right)}{y}\right)N
\\ &\cdot \left(\frac{\overline{c}}{\log N} +\frac{64 \log \frac{26}{21}}{\log^3 N} +\left(\log \left(\frac{8}{3} \right)+\frac{64}{\log^2 N} \right)\left( \frac{3\log (1+\epsilon_1))}{\log^2 N}+ \frac{27}{\log^3 N}\right)\right),
\end{align*}
with $$\overline{c}=\int_{1/8}^{1/3} \frac{\log(2-3\beta)}{\beta (1-\beta)}d \beta <0.363084.$$
\end{lem}
\begin{proof}
With $(1+\epsilon_1)N/(p_1 p_2)>p_3\ge y\ge 2657$, by \cite[Corollary 1]{Schoenfeld1976} we obtain
\begin{equation*}
\pi \left (\frac{(1+\epsilon_1)N}{p_1 p_2} \right)\le \left(1+\epsilon_1\right)\left(\frac{\log y \left(\li (y)-\sqrt{y}\log y /(8\pi)\right)}{y}\right)\frac{N}{p_1 p_2 \log (N/p_1p_2)}.
\end{equation*}
Now since $p_1 < p_2 \le p_3$ and $p_1p_2p_3<(1+\epsilon_1)N$, we have $p_1p_2^2<(1+\epsilon_1)N$ and 
\begin{equation*}
p_3 < \frac{(1+\epsilon_1)N}{p_1p_2}.
\end{equation*}
Thus, from the definition \eqref{eq:Bin} of $\overline{B}$,
\begin{align*}
\left| \overline{B}\right| &\le\sum_{\substack{z \le p_1<y \le p_2 \le p_3 \\ p_1p_2p_3 <(1+\epsilon_1)N}} 1 \le \sum_{\substack{z \le p_1<y \le p_2  \\ p_1p_2^2 <(1+\epsilon_1)N}} \pi \left (\frac{(1+\epsilon_1)N}{p_1 p_2} \right) \\ & \le \left(1+\epsilon_1\right)\left(\frac{\log y \left(\li (y)-\sqrt{y}\log y /(8\pi)\right)}{y}\right)N\sum_{z \le p_1 <y}\frac{1}{p_1}\sum_{y\le p_2<w}\frac{1}{ p_2 \log (N/p_1p_2)},
\end{align*}
with $w=\sqrt{\frac{(1+\epsilon_1)N}{p_1}}$. \par
We now introduce the functions $h_p(t)=\left(\log N/pt\right)^{-1}$ and
\begin{equation*}
H(u)=\int_y^{\sqrt{N/u}} h_u(t)d\log\log t.
\end{equation*}
Using \cite[Lemma 1 (ii)]{Greaves} and \cite[Theorem 5]{RosserSchoenfeld1975},
\begin{align*}
\sum_{y\le p_2<w}\frac{1}{ p_2 \log (N/p_1p_2)}&\le \int_y^wh_{p_1}(t) d \log \log t+\frac{h_{p_1}(w)}{\log^2 y}\\ &= H(p_1)+\int_{\sqrt{\frac{N}{p_1}}}^wh_{p_1}(t) d \log \log t+ \frac{h_{p_1}(w)}{\log^2 y}\\ &
\le H(p_1)+\frac{10\log (1+\epsilon_1)}{\log^2 N}+ \frac{27}{\log^3 N}.
\end{align*}
Where, in the last step, we used 
\begin{equation*}
\int_{\sqrt{\frac{N}{p_1}}}^wh_{p_1}(t) d \log \log t\le \frac{10\log (1+\epsilon_1)}{\log^2 N},
\end{equation*} 
obtained by the change of variables $t=\sqrt{N/p_1}s$, as done in \cite[p. 290]{Nathanson1996}, and 
\begin{equation*}
h_{p_1}(w)=\frac{2}{\log \left(\frac{N}{(1+\epsilon_1)p_1}\right)}\le \frac{3}{\log N}.
\end{equation*}
Therefore, also using \cite[Theorem 5]{RosserSchoenfeld1975},
\begin{align*}
\sum_{z \le p_1 <y}\frac{1}{p_1}&\sum_{y\le p_2<w}\frac{1}{ p_2 \log (N/p_1p_2)}\\ &
\le \sum_{z \le p_1 <y}\frac{H(p_1)}{p_1}+\left(\log \frac{8}{3} +\frac{64}{\log^2 N} \right)\left( \frac{10\log (1+\epsilon_1)}{\log^2 N}+ \frac{27}{\log^3 N}\right).
\end{align*}
We now note that $H(y)=0$ and, upon using the substitution $t=N^{\tau}$ in the definition of $H$, $H(z)\le\frac{0.56}{\log N}$.
Using \cite[Lemma 1 (ii)]{Greaves} and \cite[Theorem 5]{RosserSchoenfeld1975} and $\int_z^y H(u)d\log \log u=\frac{\overline{c}}{\log N}$ (see \cite[p. 291]{Nathanson1996}) we thus have
\begin{align*}
\sum_{z \le p_1 <y}\frac{H(p_1)}{p_1}&<\int_z^y H(u)d\log \log u+ \frac{H(z)}{\log^2 z}\\ & \le \frac{\overline{c}}{\log N} +\frac{36}{\log^3 N}.
\end{align*}
This concludes the proof of the lemma.
\end{proof}

Now let us bound the error term of \eqref{eq:SBP}. By the definition of $r_d^{(j)}$, we have
\begin{equation*}
r_d^{(j)}=\sum_{\substack{p_1p_2p_3\equiv N \pmod d \\ z\le p_1<y\le p_2\le p_3,~ w_j\le p_1<w_j(1+\epsilon_1) \\w_j p_2p_3<N,~ (p_2p_3,N)=1 }} 1-\frac{1}{\varphi (d)} \sum_{\substack{z\le p_1<y\le p_2\le p_3\\w_j\le p_1<w_j(1+\epsilon_1) \\w_j p_2p_3<N,~ (p_2p_3,N)=1 }}1.
\end{equation*}
We now add the condition $(p_1p_2p_3,d)=1$ in the second sum. This is equivalent to $(p_1,d)=1$, since the condition $(p_2p_3,d)=1$ already follows from the fact that $d$ divides $P(y)$. Thus, adding the condition $(p_1p_2p_3,d)=1$ in the second sum makes it decrease by at most 
\begin{equation*}
\frac{1}{\varphi(d)}\sum_{\substack{p_1p_2p_3<(1+\epsilon_1)N\\p_1|d,p_1 \ge z}}1\le \frac{(1+\epsilon_1)N}{\varphi(d)}\sum_{\substack{p_1|d,p_1 \ge z}}\frac{1}{p_1}\le \frac{(1+\epsilon_1)Nw(d)}{z\varphi(d)}.
\end{equation*}
We set $a(n)$ to be the characteristic function of the set of integers $n=p_2p_3$ with $y\le p_2 < p_3$ and $(N,p_2p_3)=1$. Then, by applying Lemma \ref{lem:bil}, we get for some $0 \leq \theta \le 1$
\begin{equation*}
r_d^{(j)}=\sum_{n<X}\sum_{\substack{Z \le p <Y\\np\equiv N \text{ (mod $d$)}}} a(n)-\frac{1}{\varphi (n)}\sum_{n<X}\sum_{\substack{Z\le p<Y \\ (np,d)=1}} a(n)+\frac{(1+\epsilon_1)\theta Nw(d)}{z\varphi(d)},
\end{equation*}
with
\begin{equation*}
X=\frac{N}{w_j}, \quad
Y=\min\left(y,(1+\epsilon_1)w_j\right), \quad
Z=w_j, \quad
a=N.
\end{equation*}
Let us define $D^* = \frac{\sqrt{N}}{\log^{A+1} N} \geq DQ$.
It is easy to see that with the above choices, the conditions \eqref{conditions-B} and $N\geq 4\cdot 10^{18}$, Lemma~\ref{lem:bil} holds. Therefore,
\begin{align*}
R^{(j)} \le & \sum_{\substack{d<D^*\\ d|P(y)}}\left|r_d^{(j)}\right| \\
\le & \sum_{\substack{d<D^*\\ d|P(y)}}\left|\sum_{n<\frac{N}{w_j}}\sum_{\substack{Z \le p <Y\\np\equiv N \text{ (mod $d$)}}} a(n)-\frac{1}{\varphi (n)}\sum_{n<\frac{N}{w_j}}\sum_{\substack{Z\le p<Y\\ (np,d)=1}}a(n)\right|+\sum_{d<D^*}\left(\frac{(1+\epsilon_1)Nw(d)}{z\varphi(d)} \right)
\\ &
\le m(X,Y,D^*)\frac{XY \log^{5/3} D^* \log^{1/3} Y}{Y^{1/6}}+0.762(1+\epsilon_1)N^{\frac{7}{8}}\log^2 N \\&
\le m\left(N^{2/3},(1+\epsilon_1)N^{1/8},\frac{\sqrt{N}}{\log^{A+1}N}\right)\frac{N^{47/48} (1 + \epsilon_1)^{5/6} \log^2 N}{2^{5/3}~ 3^{1/3}} + 0.762(1+\epsilon_1)N^{\frac{7}{8}}\log^2 N,
\end{align*}
for each $j=0, 1, \ldots, j_0$. Since $j_0 \leq \frac{5 \log N}{24 \log(1 + \epsilon_1)}$ by \eqref{eq: bound_j_0}, we get
\begin{align}
R&=\sum_{j=0}^{j_0} R^{(j)} \nonumber \\&
\le 0.046~ m\left(N^{2/3},(1+\epsilon_1)N^{1/8},\frac{\sqrt{N}}{\log^{A+1}N}\right)\frac{N^{47/48} (1 + \epsilon_1)^{5/6} \log^3 N}{\log(1+\epsilon_1)} + 0.159\frac{(1+\epsilon_1)}{\log(1 + \epsilon_1)}N^{\frac{7}{8}}\log^3 N \nonumber \\&
\le \frac{N^{47/48} (1 + \epsilon_1) \log^3 N}{\log(1+\epsilon_1)} \Bigg[
0.046 \cdot m\left(N^{2/3},(1+\epsilon_1)N^{1/8},\frac{\sqrt{N}}{\log^{A+1}N}\right) (1+\epsilon_1)^{-1/6} + 0.159 N^{-\frac{5}{48}} \Bigg]. \label{eq:Rfin}
\end{align}
The desired result now follows from \eqref{eq:SBP}, Lemma~\ref{lemma:Bover} and \eqref{eq:Rfin}.

\subsection{Proof of Theorem \ref{Theo:B.}}
\label{subsection:theo}
We combine Lemma \ref{lem:>} with the bounds obtained in Theorems \ref{thm: lower-S-A}, \ref{thm: upper-S-A-q}, and \ref{theo:B}. We choose the parameters $N \geq X_2 = \exp(\exp(15.85))$, $u = 10^{4.27}$, $\alpha_1 = \alpha_2 = \alpha_3 = 10^{-2.61}$, $A = 4$, $\epsilon_1 = 10^{-20}$, $\epsilon = \frac{1}{8137}$, $C_1 = 112$, and $C_2 = 114$. All computations were implemented in SageMath 9.3.

\subsection{Proof of Corollary \ref{distinctcor}}
\label{subsection:cor}
Let $\pi'_2(N)$ be a number of representations of $N$ as a sum of a prime number and a square-free semi-prime number. Then 
$$\pi'_2(N) \geq \pi_2(N) - \sqrt{N} - 1,$$
where $1$ counts the representation $N = p + 1$, if such exists, and $\sqrt{N}$ bounds the number integers up to $N$ with two equal prime factors.

With the choice of parameters from the previous section we obtain:
$$\pi'_2(N) \geq 4 \cdot 10^{-4} \frac{U_N N}{\log^2 N} - \sqrt{N} - 1 > 0,$$
the last inequality being checked using SageMath 9.3.

\bibliography{refs.bib}

\providecommand{\bysame}{\leavevmode\hbox to3em{\hrulefill}\thinspace}
\providecommand{\MR}{\relax\ifhmode\unskip\space\fi MR }
\providecommand{\MRhref}[2]{%
  \href{http://www.ams.org/mathscinet-getitem?mr=#1}{#2}
}
\providecommand{\href}[2]{#2}
\begin{thebibliography}{10}

\bibitem{BJS22}
M.~Bordignon, D.R. Johnston, and V.~Starichkova, \emph{An explicit version of
  {C}hen's theorem}, arxiv:2207.09452 (2022).

\bibitem{BroadbentSamuelKadiri2021}
S.~Broadbent, H.~Kadiri, A.~Lumley, N.~Ng, and K.~Wilk, \emph{Sharper bounds
  for the {C}hebyshev function {$\theta(x)$}}, Math. Comp. \textbf{90} (2021),
  no.~331, 2281--2315.

\bibitem{Buthe1}
J.~B\"uthe, \emph{A {B}run-{T}itchmarsh inequality for weighted sums over prime
  numbers}, Acta Arith. \textbf{166} (2014), no.~3, 289--299.

\bibitem{Buthe2018}
J.~B\"{u}the, \emph{An analytic method for bounding {$\psi(x)$}}, Math. Comp.
  \textbf{87} (2018), no.~312, 1991--2009.

\bibitem{CH_L_21}
M.~Cully-Hugill and E.~S. Lee, \emph{Explicit interval estimates for prime
  numbers}, Preprint on arXiv:2103.05986 (2021), Submitted.

\bibitem{Davenport}
H.~Davenport, \emph{Multiplicative {N}umber {T}heory}, Graduate Texts in
  Mathematics, vol.~74, Springer-Verlag, New York, 2000, Third Edition.

\bibitem{Des}
J.-M. Deshouillers, G.~Effinger, H.~te~Riele, and D.~Zinoviev, \emph{A complete
  {V}inogradov {$3$}-primes theorem under the {R}iemann hypothesis}, Electron.
  Res. Announc. Amer. Math. Soc. \textbf{3} (1997), 99--104. \MR{1469323}

\bibitem{ErnvallPalojarvi2020}
A.-M. Ernvall-Hyt\"{o}nen and N.~Paloj\"{a}rvi, \emph{Explicit bound for the
  number of primes in arithmetic progressions assuming the generalized
  {R}iemann hypothesis}, Math. Comp. \textbf{91} (2022), no.~335, 1317--1365.

\bibitem{Greaves}
G.~Greaves, \emph{Sieves in {N}umber {T}heory}, Ergebnisse der Mathematik und
  ihrer Grenzgebiete (3) [Results in Mathematics and Related Areas (3)],
  vol.~43, Springer-Verlag, Berlin, 2001. \MR{1836967}

\bibitem{Helf}
H.~Helfgott, \emph{The ternary {G}oldbach problem}, arXiv:1501.05438 (2015).

\bibitem{KadiriLumley2014}
H.~Kadiri and A.~Lumley, \emph{Short effective intervals containing primes},
  Integers \textbf{14} (2014), Paper No. A61, 18.

\bibitem{Kan}
L.~Kaniecki, \emph{On \v{S}nirelman's constant under the {R}iemann hypothesis},
  Acta Arith. \textbf{72} (1995), no.~4, 361--374. \MR{1348203}

\bibitem{Nathanson1996}
M.B. Nathanson, \emph{Additive {N}umber {T}heory}, Graduate Texts in
  Mathematics, vol. 164, Springer-Verlag, New York, 1996, The classical bases.

\bibitem{O_H_P_14}
T.~Oliveira~e Silva, S.~Herzog, and S.~Pardi, \emph{Empirical verification of
  the even {G}oldbach conjecture and computation of prime gaps up to {$4\cdot
  10^{18}$}}, Math. Comp. \textbf{83} (2014), no.~288, 2033--2060.

\bibitem{Guy}
G.~Robin, \emph{Estimation de la fonction de {T}chebychef $\theta$ sur le
  $k$-ième nombre premier et grandes valeurs de la fonction $\omega(n)$ nombre
  de diviseurs premiers de $n$}, Acta Arith. \textbf{42} (1983), no.~4,
  367--389.

\bibitem{RosserSchoenfeld1962}
J.B. Rosser and L.~Schoenfeld, \emph{Approximate formulas for some functions of
  prime numbers}, Illinois Journal of Mathematics \textbf{6} (1962), no.~1,
  64--94.

\bibitem{RosserSchoenfeld1975}
J.B. Rosser and L.~Schoenfeld, \emph{Sharper bounds for the {C}hebyshev
  functions $\theta(x)$ and $\psi(x)$}, Math. Comp. \textbf{29} (1975),
  no.~129, 243--269.

\bibitem{Schoenfeld1976}
L.~Schoenfeld, \emph{Sharper bounds for the {C}hebyshev functions $\theta(x)$
  and $\psi(x)$. {II}}, Math. Comp. \textbf{30} (1976), no.~134, 337--360.

\bibitem{Vin}
I.M. Vinogradov, \emph{Representation of an odd number as a sum of three
  primes}, Dokl.Akad. Nauk. SSR \textbf{15} (1937), 291--294.

\bibitem{Yamada}
T.~Yamada, \emph{Explicit {C}hen's theorem}, https://arxiv.org/abs/1511.03409
  (2015).

\end{thebibliography}
\bibliographystyle{amsplain}
\end{document}